\documentclass[12pt,reqno]{amsart}

\usepackage[arrow,matrix,curve]{xy}

\usepackage[dvips]{graphicx} 

\usepackage{amssymb, latexsym, amsmath, amscd, array, hyperref,
setspace    
}

\usepackage{breakurl}   

\newtheorem{theorem}{Theorem}[section]

\newtheorem{proposition}[theorem]{Proposition}

\theoremstyle{definition}
\newtheorem{definition}[theorem]{Definition}

\newtheorem{remark}[theorem]{Remark}

\numberwithin{equation}{section}

\newcommand\N {{\mathbb N}} 

\newcommand\R {{\mathbb R}}
\newcommand\Q {{\mathbb Q}}
 
\newcommand\st{{\textbf{st}}} 
\newcommand{\hr} {{{}^{\mathfrak{h}}\hspace*{-0.4pt}\R}}

\newcommand\astr{{{}^{\ast}\hspace{-0.5pt}\R}}
\newcommand\astq{{{}^{\ast}\Q}}
\newcommand\astf{{{}^{\ast}\hspace*{-2.3pt}F}}
\newcommand\astff{{}^\ast\hskip-2.9pt f}
\newcommand\asta{{{}^\ast\hspace{-3.5pt}A}}
\newcommand\astn{{}^{\ast}\hspace{-1pt}\N}
\newcommand\astp{{{}^{\ast}\hspace{-0.5pt}\mathbb{P}}}
\newcommand{\ha}{\mathfrak{h}}
\newcommand\astu{{}^{\ast}\hspace{-1.14pt}U}
\newcommand\astx{{}^{\ast}\hspace{-3pt}X}
\newcommand\asts{{}^{\ast}\hspace{-1.7pt}S}

\author[P. B\l aszczyk]{Piotr B\l{}aszczyk}\address{P. B\l{}aszczyk,
Institute of Mathematics, Pedagogical University of Cracow,
Poland}\email{pb@up.krakow.pl}

\author [V. Kanovei] {Vladimir Kanovei} \address{V. Kanovei, IPPI,
Moscow, and MIIT, Moscow, Russia}\email{kanovei@googlemail.com}

\author [M. Katz] {Mikhail G. Katz}\address{M. Katz, Department of
Mathematics, Bar Ilan University, Ramat Gan 52900
Israel}\email{katzmik@macs.biu.ac.il}

\author[T. Nowik]{Tahl Nowik}\address{T. Nowik, Department of
Mathematics, Bar Ilan University, Ramat Gan 52900
Israel}\email{tahl@math.biu.ac.il}


\begin{document}

\thispagestyle{empty}


\title{Monotone subsequence via ultrapower}

\begin{abstract}
An ultraproduct can be a helpful organizing principle in presenting
solutions of problems at many levels, as argued by Terence Tao.  We
apply it here to the solution of a calculus problem: every infinite
sequence has a monotone infinite subsequence, and give other
applications.

\end{abstract}

\maketitle 


\section{Introduction}

Solutions to even elementary calculus problems can be tricky but in
many cases, enriching the foundational framework available enables one
to streamline arguments, yielding proofs that are more natural than
the traditionally presented ones.

We explore various proofs of the elementary fact that every infinite
sequence has a monotone infinite subsequence, including some that
proceed without choosing a convergent one first.

An ultraproduct can be a helpful organizing principle in presenting
solutions of problems at many levels, as argued by Terence Tao in
\cite{Ta14}.  We apply it here to the solution of the problem
mentioned above.  A related but \emph{different} problem of proving
that every infinite totally ordered set contains a monotone sequence
is treated by Hirshfeld in \cite[Exercise\;1.2, p.\;222]{Hi88}.  We
first present the ultrapower construction in Section~\ref{s3}.
Readers familiar with ultraproducts can skip ahead to the proof in
Section~\ref{s2}.

\section{Ultrapower construction}
\label{s3}

Let us outline a construction (called an \emph{ultrapower}) of a
hyperreal extension~$\R\hookrightarrow\astr$ exploited in our solution
in Section~\ref{s2}.  Let~$\R^{\N}$ denote the ring of sequences of
real numbers, with arithmetic operations defined termwise.  Then we
have a totally ordered field~$\astr=\R^{\N}\!/\text{MAX}$ where
``MAX'' is a suitable maximal ideal.  Elements of~$\astr$ are called
hyperreal numbers.  Note the formal analogy between the quotient
$\astr=\R^{\N}\!/\text{MAX}$ and the construction of the real numbers
as equivalence classes of Cauchy sequences of rational numbers.  In
both cases, the subfield is embedded in the superfield by means of
constant sequences, and the ring of sequences is factored by a
\emph{maximal ideal}.  

We now describe a construction of such a maximal ideal
$\text{MAX}\subseteq\R^\N$ exploiting a suitable finitely additive
measure~$\xi\colon\mathcal{P}(\N)\to\{0,1\}~$ (thus~$\xi$ takes only
two values,~$0$ and~$1$) taking the value~$1$ on each cofinite set,%
\footnote{For each pair of complementary \emph{infinite} subsets
of~$\N$, such a measure~$\xi$ ``decides'' in a coherent way which one is
``negligible'' (i.e., of measure~$0$) and which is ``dominant''
(measure~$1$).}
where~$\mathcal{P}(\N)$ is the set of subsets of~$\N$.  The ideal MAX
consists of all ``negligible'' sequences~$\langle u_n\rangle$, i.e.,
sequences which vanish for a set of indices of full measure~$\xi$,
namely, $\xi\big(\{n\in\N\colon u_n=0\}\big)=1$.  The
subset~$\mathcal{U}=\mathcal{U}_\xi\subseteq\mathcal{P}(\N)$
consisting of sets of full measure~$\xi$ is called a free ultrafilter
(these can be shown to exist using Zorn's lemma).  A similar
construction applied to~$\Q$ produces the field~$\astq$ of
hyperrational numbers.  The construction can also be applied to a
general ordered set~$F$ to obtain an ultrapower extension denoted
$\astf=F^{\N}\!/\mathcal{U}$.

\begin{definition}
The \emph{order} on the field~$\astf$ is defined by setting
\[
[\langle u_n\rangle]<[\langle v_n\rangle] \text{\; if and only if \;}
\xi(\{n\in\N\colon u_n<v_n\})=1
\]
or equivalently~$\{n\in\N\colon u_n<v_n\}\in\mathcal{U}$.
\end{definition}

In particular, every element $x\in F$ is canonically identified with
the class $[\langle x\rangle]$ of the constant sequence $\langle
x\rangle$ with general term $x$.  Then $x\in\astf$ satisfies~$x<v$ if
and only if~$\{n\in\N\colon x<v_n\}\in\mathcal{U}$.

\section{Solution}
\label{s2}

Let~$F$ be an ordered field.  We are mainly interested in the
cases~$F=\Q$ and~$F=\R$ though the arguments go through in greater
generality for an arbitrary totally ordered set.

\begin{theorem}
A sequence~$\langle u_n\rangle$ of elements of~$F$ necessarily
contains a subsequence~$\left\langle u_{n_k}\right\rangle$ such that
either~$u_{n_k}\geq u_{n_\ell}$ whenever~$k>\ell$, or~$u_{n_k}\leq
u_{n_\ell}$ whenever~$k>\ell$.
\end{theorem}

This is an immediate consequence of the following more detailed
result.

\begin{theorem}
\label{t32}
Let~$u\in\astf=F^\N\!/\mathcal{U}$ be the element obtained as the
equivalence class of the sequence~$\langle u_n\rangle$.  Consider the
partition $\N=A\sqcup B\sqcup C$ where~$A=\{n\in\N\colon
u_n<u\}$,~$B=\{n\in\N\colon u_n=u\}$,~$C=\{n\in\N\colon u_n >u\}$.
Then exactly one of the following three possibilities occurs:
\begin{enumerate}
\item
$B\in\mathcal{U}$ and then~$\langle u_n\rangle$ contains an infinite constant
subsequence;
\item
$A\in\mathcal{U}$ and then~$\langle u_n\rangle$ contains an infinite
strictly increasing sub\-sequence;
\item
$C\in\mathcal{U}$ and then~$\langle u_n\rangle$ contains an infinite strictly
decreasing subsequence.
\end{enumerate}
\end{theorem}

\begin{proof}
By the property of an ultrafilter, exactly one of the sets~$A,B,C$ is
in~$\mathcal{U}$.  If~$B\in\mathcal{U}$ then~$u$ is an element of the
subfield~$F\subseteq\astf$ (embedded via constant sequences).
Since~$B\subseteq\N$ is necessarily infinite, enumerating it we obtain
the desired subsequence.

Now assume~$A\in\mathcal{U}$.  We choose any element~$u_{n_1}\in A$ to
be the first term in the subsequence.  We then inductively choose the
index~$n_{k+1}>n_k$ in~$A$ so that~$u_{n_{k+1}}$ is the earliest term
greater than~$u_{n_k}$ and therefore closer to~$u$ than the previous
term~$u_{n_{k}}$.  If the subsequence were to terminate at, say,
$u_p$, this would imply that~$\{n\in\N\colon u_n\leq u_p\}\in\mathcal
U$ and therefore~$u\leq u_p$, contradicting the definition of the
set~$A$.  Therefore we necessarily obtain an infinite increasing
subsequence.

The case~$C\in\mathcal{U}$ is similar and results in a decreasing
sequence.
\end{proof}

\begin{remark}
\label{r33}
The proof is essentially a two-step procedure: (1)~we plug the
sequence into the ultrapower construction, producing an
element~$u\in\astf$; (2)~in each of the cases specified by the element
$u$, we inductively find a monotone subsequence.
\end{remark}

The approach exploiting~$\astf$ has the advantage that the proof does
not require constructing a completion of the field in the case~$F=\Q$.
To work with the ultrapower, one needs neither advanced logic nor a
crash course in NSA, since the ultrapower construction involves merely
quotienting by a maximal ideal as is done in any serious undergraduate
algebra course (see Section~\ref{s3}).

A monotone sequence can also be chosen by the following more
traditional consideration.  If the sequence is unbounded, one can
choose a sequence that diverges to infinity.  If the sequence is
bounded, one applies the Bolzano-Weierstrass theorem (\emph{each
bounded sequence has a convergent subsequence}) to extract a
convergent subsequence.  Finally, a convergent sequence contains a
monotone one by analyzing the terms lying on one side of the limit
(whichever side has infinitely many terms).

The proof via an ultrapower allows one to bypass the issue of
convergence.  Once one produces a monotone subsequence, it will also
be convergent in the bounded case but only when the field is complete.
Furthermore one avoids the use of the Bolzano--Weierstrass theorem.

Since in the case of~$F=\Q$ the Bolzano--Weierstrass theorem is
inapplicable, one would need first to complete~$\Q$ to~$\R$ by an
analytic procedure which is arguably at least as complex as the
algebraic construction involved in the ultrapower of Section~\ref{s3}.

There is a clever proof of the same result, as follows (see e.g.,
problem\;6 on page\;4 in Newman \cite{Ne82}).  Call a term in the
sequence a \emph{peak} if it is larger than everything which comes
after it.  If there are infinitely many peaks, they form an infinite
decreasing subsequence.  If there are finitely many peaks, start after
the last one.  From here on every term has a larger term after it, so
one inductively forms an increasing subsequence (from this lemma one
derives a simple proof of the Bolzano--Weierstrass theorem).

\begin{remark}
The proof in Newman consists of two steps: (1) introduce the idea of a
peak; (2) consider separately the cases when the number of peaks is
finite or infinite to produce the desired monotone subsequence.  While
the basic structure of the proof is similar to that using the
ultrapower (see Remark~\ref{r33}), the basic difference is that
step~(1) in Newman is essentially ad-hoc, is tailor-made for this
particular problem, and is not applicable to solving other problems.
Meanwhile the ultrapower construction is applicable in many other
situations (see e.g., Section~\ref{s4}).
\end{remark}

While the proof in Newman does not rely on an ultrapower, the idea of
the ultrapower proof is more straightforward once one is familiar with
the ultrapower construction, since it is natural to plug a sequence
into it and examine the consequences.

We provide another illustration of how the element~$u=[\langle
u_n\rangle]$ can serve as an organizing principle that allows us to
detect properties of monotone subsequences.  To fix ideas let~$F=\R$.
An element~$u\in\astr$ is called finite if~$-r<u<r$ for a
suitable~$r\in\R$.  Let~$\hr\subseteq\astr$ be the subring of finite
elements of~$\astr$.  The standard part function~$\st\colon\hr\to\R$
rounds off each finite hyperreal~$u$ to its nearest real number
$u_0=\st(u)$.

\begin{proposition}
If~$u\in\hr$ and~$u>u_0$ then the sequence~$\langle u_n\rangle$
possesses a strictly decreasing subsequence.
\end{proposition}

\begin{proof}
Since~$u>u_0$ we have~$\{n\in\N\colon u_n>u_0\}\in\mathcal{U}$.  We
start with an arbitrary $n_1\in\{n\in\N\colon u_n>u_0\}$ and
inductively choose $n_{k+1}$ so that $u_{n_{k+1}}$ is closer to $u$
than $u_{n_k}$.  We argue as in the proof of Theorem~\ref{t32} to show
that the process cannot terminate and therefore produces an infinite
subsequence.
\end{proof}

\section{Compactness}
\label{s4}

A more advanced application is a proof of the nested decreasing
sequence property for compact sets (Cantor's intersection theorem)
using the property of \emph{saturation}.  Such a proof exbibits
compactness as closely related to the more general property of
saturation, shedding new light on the classic property of compactness.

A typical proof of \emph{Cantor's intersection theorem} for a nested
decreasing sequence of compact subsets $A_n\subseteq\R$ would use the
monotone sequence $\langle u_n\rangle$ where $u_n$ is the minimum of
each $A_n$.  We will present a different and more conceptual proof.

Each set~$A\subseteq\R$ has a \emph{natural extension} denoted
$\asta\subseteq\astr$.  Similarly the powerset $P=\mathcal P(\R)$ has
a natural extension~$\astp$ identified with a proper subset of
$\mathcal P(\astr)$.  Each element of~$\astp$ is naturally identified
with a subset of~$\astr$ called an \emph{internal set}.

The principle of \emph{saturation} holds for arbitrary nested
decreasing sequences of internal sets but we will present it in a
following special case.

\begin{theorem}[Saturation]
\label{t971}
If~$\langle A_n \colon n\in\N\rangle$ is a nested decreasing sequence
of nonempty subsets of\,~$\R$ then the sequence~$\langle\asta_n\colon
n\in\N\rangle$ has a common point.
\end{theorem}

\begin{proof}
Let~$\mathbb {P}=\mathcal{P}(\R)$ be the set of subsets of~$\R$.  We
view the sequence~$\langle A_n\in\mathbb{P}\colon n\in\N\rangle$ as a
function~$f\colon\N\to\mathbb P,\; n\mapsto A_n$.  By the extension
principle we have a function~$\astff \colon \astn\to \astp$.
Let~$B_n=\astff(n)$.  For each finite~$n$ we
have~$B_n=\asta_n\in\astp$.  For each infinite value of the
index~$n=H$ the entity~$B_H\in\astp$ is by definition internal but is
not (necessarily) the natural extension of any subset of~$\R$.

If~$\langle A_n \rangle$ is a \emph{nested} sequence in~$\mathbb P$
then by transfer~$\langle B_n\colon n\in\astn \rangle$ is a nested
sequence in~$\astp$ with each $B_n$ nonempty.  Let~$H$ be a fixed
infinite index.  Then for each finite~$n$ the
set~$\asta_n\subseteq\astr$ includes~$B_H$.  Choose any element~$c\in
B_H$.  Then~$c$ is contained in~$\asta_n$ for each finite~$n$ so that
$c\in \bigcap_{n\in\N} \asta_n$ as required.
\end{proof}

\begin{remark}
\label{r1013}
An equivalent formulation of Theorem~\ref{t971} is as follows.  If the
family of subsets~$\{A_n\}_{n\in\N}$ has the finite intersection
property then~$\exists c\in\bigcap_{n\in\N} \asta_n$.
\end{remark}

Let~$X$ be a topological space.  Let~$p \in X$.  The \emph{halo}
of~$p$, denoted~$\ha(p)$ is the intersection of all~$\astu$ where~$U$
runs over all neighborhoods of~$p$ in~$X$ (a neighborhood of~$p$ is an
open set that contains~$p$).  A point~$y \in \astx$ is called
\emph{nearstandard} in~$X$ if there is~$p \in X$ such
that~$y\in\ha(p)$.

\begin{theorem}
\label{s43}
A space~$X$ is compact if and only if every~$y\in\astx$ is
nearstandard in~$X$.
\end{theorem}

\begin{proof}
To prove the direction~$\Rightarrow$, assume~$X$ is compact, and let
$y\in\astx$.  Let us show that~$y$ is nearstandard (this direction
does not require saturation).  Assume on the contrary that~$y$ is not
nearstandard.  This means that it is not in the halo of any
point~$p\in X$.  This means that every~$p \in X$ has a
neighborhood~$U_p$ such that~$y\not\in\astu_p$. The
collection~$\{U_p\}_{p\in X}$ is an open cover of~$X$.  Since~$X$ is
compact, the collection has a finite
subcover~$U_{p_1},\ldots,U_{p_n}$, so that~$X=U_{p_1}\cup \ldots \cup
U_{p_n}$. But for a finite union, the star of union is the union of
stars.  Thus~$\astx$ is the union of~$\astu_{p_1},\ldots,\astu_{p_n}$,
and so the point~$y$ is in one of the
sets~$\astu_{p_1},\ldots,\astu_{p_n}$, a contradiction.

Next we prove the direction~$\Leftarrow$ (this direction exploits
saturation).  Assume every~$y \in \astx$ is nearstandard, and
let~$\{U_a\}$ be an open cover of~$X$.  We need to find a finite
subcover.

Assume on the contrary that the union of any finite collection
of~$U_a$ is not all of~$X$.  Then the complements of~$U_a$ are a
collection of (closed) sets~$\{S_a\}$ with the finite intersection
property.  It follows that the collection~$\{\asts_a\}$ similarly has
the finite intersection property.  By saturation (see
Remark~\ref{r1013}), the intersection of all~$\asts_a$ is non-empty.
Let~$y$ be a point in this intersection.  Let~$p \in X$ be such
that~$y \in \ha(p)$.  Now~$\{U_a\}$ is a cover of~$X$ so there is
a~$U_b$ such that~$p\in U_b$.  But~$y$ is in~$\asts_a$ for all~$a$, in
particular~$y \in\asts_b$, so it is not in~$\astu_b$, a contradiction
to~$y \in\ha(p)$.
\end{proof}


\begin{theorem}[Cantor's intersection theorem]
A nested decreasing sequence of nonempty compact sets has a common
point.
\end{theorem}
 
\begin{proof}
Given a nested sequence of compact sets $K_n$, we consider the
corresponding decreasing nested sequence of internal sets,
$\langle{}^\ast\!K_n\colon n\in\mathbb N\rangle$.  This sequence has a
common point $x$ by saturation.  But for a compact set $K_n$, every
point of ${}^\ast\!K_n$ is nearstandard (i.e., infinitely close to a
point of $K_n$) by Theorem~\ref{s43}.  In particular, $\st(x)\in K_n$
for all $n$, as required.
\end{proof}

More advanced applications can be found in \cite{17f, 18h, NK}.

\end{document}